\documentclass[12pt]{amsart}

\usepackage{cancel}
\usepackage{latexsym, amssymb, amsmath}
\usepackage{soul}
\usepackage{amsfonts, graphicx}
\usepackage{graphicx,color}
\newcommand{\be}{\begin{equation}}
\newcommand{\ee}{\end{equation}}
\newcommand{\beq}{\begin{eqnarray}}
\newcommand{\eeq}{\end{eqnarray}}

\newtheorem{thm}{Theorem}[section]

\newtheorem{lma}{Lemma}[section]
\newtheorem{prop}{Proposition}[section]
\newtheorem{cor}{Corollary}[section]
\newtheorem{defn}{Definition}[section]
\theoremstyle{remark}
\newtheorem{rem}{Remark}[section]
\numberwithin{equation}{section}

\def\be{\begin{equation}}
\def\ee{\end{equation}}
\def\bee{\begin{equation*}}
\def\eee{\end{equation*}}

\def\bs{\begin{split}}
\def\es{\end{split}}

\def\ol{\overline}
\def\lf{\left}
\def\ri{\right}

\def\K{K\"ahler }

\def\KR{K\"ahler-Ricci }
\def\Ric{\text{\rm Ric}}
\def\Rm{\text{\rm Rm}}

\def\p{\partial}

\def\ol{\overline}
\def\heat{\lf(\frac{\p}{\p t}-\Delta\ri)}

\def\e{\varepsilon}
\def\a{{\alpha}}

\def\ijb{{i\bar{j}}}

\begin{document}

\title[]
{Hermitian manifolds with quasi-negative curvature}

 \author{Man-Chun Lee}
\address[Man-Chun Lee]{Department of Mathematics,Northwestern University,2033 Sheridan Road, Evanston, IL 60208}
\email{mclee@math.northwestern.edu}

\renewcommand{\subjclassname}{
  \textup{2010} Mathematics Subject Classification}

\date{\today}

\begin{abstract}
In this work, we show that along a particular choice of Hermitian curvature flow, the non-positivity of the first Ricci curvature will be preserved if the initial metric has Griffiths non-positive Chern curvature. If in addition, the first Ricci curvature is negative at a point, then the canonical line bundle is ample.
\end{abstract}


\maketitle

\markboth{Man-Chun Lee}{Hermitian manifolds with quasi-negative  curvature}
\section{introduction}
Let $(M^n,g)$ be a compact Hermitian manifold with complex dimension $n$. In Hermitian differential geometry, a weaker version of Kobayashi conjecture predicts that the canonical line bundle is ample if the $g$ has negative holomorphic sectional curvature. When $g$ is in addition K\"ahler, it was recently proved by  Wu-Yau \cite{WuYau2016} in the projective case and Tosatti-Yang \cite{TosattiYang2017} in the \K case. Later on, Diverio and Trapani \cite{DiverioTrapani2016} further generalized the result which only assume holomorphic sectional curvature to be non-positive on $M$ and negative at some points. In this case, we say that the holomorphic sectional curvature is quasi-negative. In \cite{WuYau2016-2}, Wu and Yau also gave a direct proof of this result. For related works, we refer readers to \cite{HeierLuWong2010,HeierLuWong2016,
HeierLuWongZheng2017,
YangZheng2016,LeeStreets2019}.

Inspired by the Kobayashi conjecture and the above mentioned results, it was conjectured that the same conclusion holds if $g$ is only assumed to be Hermitian metric with quasi-negative holomorphic sectional curvature, for example see \cite{YangZheng2016}. To the best of author's knowledge, if the K\"ahlerity is a priori unknown, then the conjecture is still widely open even under a slightly stronger curvature condition, \textit{real bisectional curvature} introduced in \cite{YangZheng2016}. Motivated by this, we consider Hermitian manifolds with quasi-negative bisectional curvature.

Motivated by Hamilton's tensor maximum principle along Ricci flow, we attempt to deform the Hermitian metric so that its first Ricci curvature behaves in a better way. In \cite{Liu2014}, building on the work by B\"ohm and Wilking \cite{BohmWilking2007}, Liu showed that in \K case, the \KR flow will preserve the non-positivity of Ricci curvature for a short time if the initial metric has non-positive bisectional curvature. On the non-negative side, the preservation of non-negative bisectional curvature was proved by Bando \cite{Bando} in complex dimension $3$ and by Mok \cite{Mok} in higher dimension. Thanks to the strong maximum principle, this enables them to classify manifolds with curvature with weak sign. Along this line, in the non-\K case, it is rather natural to consider the Hermitian curvature flow introduced by Streets and Tian \cite{StreetTian2011}. This is a family of Hermitian metrics $g(t)$ which satisfies 
$$\partial_t g_{i\bar j}=-S_{i\bar j} +Q_{i\bar j}$$
where $Q$ is some contraction of torsion tensor $T$. In particular in \cite{Yury2016}, Ustinovskiy discovered that when $Q_{i\bar j}=\frac{1}{2}g^{k\bar l}g^{p\bar q}T_{\bar l \bar qi}T_{kp\bar j}$, the corresponding Hermitian metric flow $g(t)$ will preserve the non-negativity of bisectional curvature which generalized the work in \cite{Mok} to non-\K setting.

In this paper, we consider the case when $Q\equiv 0$. By combining the technique of Ustinovskiy \cite{Yury2016} and a version of Uhlenbeck's trick, we show that a Hermitian analogy of curvature conditions considered by B\"ohm-Wilking \cite{BohmWilking2007} will be preserved for a short time by this  choice of Hermitian curvature flow.
\begin{thm}\label{main}
Suppose $(M,g_0)$ is a compact Hermitian manifold with Griffiths non-positive Chern curvature (see Definition \ref{BK-def}). Let $g(t),t\in [0,T]$ be a solution to the equation
\be
\left\{
  \begin{array}{ll}
   &\partial_t g(t)=    -S(g(t)); \\
    &g(0)=  g_0.
  \end{array}
\right.
\ee
where $S$ is the second Ricci curvature with respect to the Chern connection. Then there is $\tau>0$ such that the first Ricci curvature $Ric(g(t))$ is non-positive on $[0,\tau]$. 
\end{thm}
For a more detailed result, we refer to Proposition \ref{preserve-Ric}. With the help of strong maximum principle of tensor, we have the following main result.

\begin{cor}\label{main-2}
Under the assumption in Theorem \ref{main}, if in addition $g_0$ has negative first Ricci curvature at one point, then the canonical line bundle is ample. In particular, there is a K\"ahler-Einstein metric $g=-Ric(g)$ on $M$.
\end{cor}

In the non-\K Hermitian geometry,  there are various types of Ricci curvature associated to the Chern-curvature tensor due to the presence of torsion. The first Ricci curvature (or called Chern-Ricci curvature) and the second Ricci curvature is particularly  interesting. In view of the curvature assumptions, it seems to be  more natural to consider the Hermitian curvature flow introduced in \cite{StreetTian2011} rather than the Chern-Ricci flow introduced by Gill \cite{Gill2011}. Relatively, the Chern-Ricci flow seems to be more sensitive to the information from Bott-Chern class rather than differential geometric quantities. 

The paper is organized as follows: In section \ref{background}, we will collect some useful information about the Chern connection. In section 3, we will derive evolution equations for the Hermitian curvature flow. In section 4, we prove the preservation of non-positivity of first Ricci curvature, $Ric(g(t))$. In section 5, we give a proof on Corollary \ref{main-2}.

{\it Acknowledgement}: The author would like to thank Professor Jeffrey Streets for his interest in this work. He would like to thank Xiaokui Yang for pointing out the related results in  \cite{BCHM2010}. Lastly, he would like to thank Yury Ustinovskiy for patiently answering his questions.

\section{Chern connection}\label{background}

In this section, we collect some useful formulas for the Chern connection. Those materials can be found in \cite{TosattiWeinkove2015,ShermanWeinkove2013}. Let $(M,g)$ be a Hermitian manifold. The {\it Chern connection} of $g$ is defined as follows: In local holomorphic coordinates $z^i$,
 for a vector field $X^i\p_i$ where $\p_i:=\frac{\p}{\p z^i}$, $\p_{\bar i}=\frac{\p}{\p \bar z^i}$,
 $$\nabla_iX^k=\p_{i}X^k+\Gamma_{ij}^kX^j;\ \nabla_{\bar i}X^k=\p_{\bar i}X^k.
 $$
 For a $(1,0)$ form $a=a_idz^i$,
 $$
 \nabla_ia_j=\p_i a_j-\Gamma_{ij}^ka_k; \ \nabla_{\bar i}a_j=\p_{\bar i}a_j.
 $$
 Here $\nabla_i:=\nabla_{\p_i}$, etc. $\Gamma$ are the coefficients of $\nabla$,
with
$$\Gamma_{ij}^k=g^{k\bar l}\partial_i g_{j\bar l}.$$
Note that Chern connection is a connection such that $\nabla g=\nabla J=0$ and the torsion has no $(1,1)$ component.   The {\it torsion} of $g$ is defined to be
$$T_{ij}^k=\Gamma_{ij}^k-\Gamma_{ji}^k.$$
We remark that $g$ is \K if and only if $T\equiv 0$. Define the \textit{Chern curvature tensor} of $g$ to be
$$R_{i\bar jk}\,^l=-\partial_{\bar j}\Gamma_{ik}^l.$$
We raise and lower indices by using metric $g$. {Direct computations show:
$$
\ol{R_{i\bar jk\bar l}}=R_{j\bar il\bar k}.
$$}

The Chern-Ricci curvature (first Ricci curvature) is defined by
$$R_{i\bar j}=g^{k\bar l}R_{i\bar j k\bar l}=-\partial_i \partial_{\bar j}\log \det g.$$

And the second Ricci curvature is defined by $S_{i\bar j}=g^{k\bar l}R_{ k\bar li\bar j}$. Note that if $g$ is not K\"ahler, then $R_{i\bar j}$ is not necessarily equal to $S_{i\bar j}$.

\begin{lma}\label{Ricci-id}
The commutation formulas for the Chern curvature are given by
\begin{align*}
[\nabla_i,\nabla_{\bar j}]X^l=R_{i\bar j k}\,^l X^k,\quad\quad [\nabla_i,\nabla_{\bar j}]a_k=-R_{i\bar j k}\,^l a_l;\\
[\nabla_i,\nabla_{\bar j}]X^{\bar l}=-R_{i\bar j}\,^{\bar l}\,_{\bar k} X^{\bar k},\quad\quad [\nabla_i,\nabla_{\bar j}]a_{\bar k}=R_{i\bar j}\,^{\bar l}\,_{\bar k} a_{\bar l}.
\end{align*}
\end{lma}

When $g$ is not K\"ahler, the Bianchi identities may fail. The failure can be measured by the torsion tensor.

\begin{lma}\label{l-Chern-connection-1}
In a holomorphic local coordinates, let $T_{ij\bar k}=g_{p\bar k}T_{ij}^p$,  we have
\begin{align*}
R_{i\bar jk\bar l}-R_{k\bar ji\bar l}&=-\nabla_{\bar j}T_{ik\bar l},\\
R_{i\bar jk\bar l}-R_{i\bar lk\bar j}&=-\nabla_{i}T_{\bar j\bar lk},\\
R_{i\bar jk\bar l}-R_{k\bar li\bar j}&=-\nabla_{\bar j}T_{ik\bar l}-\nabla_kT_{\bar j\bar li}=-\nabla_iT_{\bar j\bar lk}-\nabla_{\bar l}T_{ik\bar j},\\
\nabla_pR_{i\bar jk\bar l}-\nabla_iR_{p\bar jk\bar l}&=-T_{pi}^rR_{r\bar jk\bar l},\\
\nabla_{\bar q}R_{i\bar jk\bar l}-\nabla_{\bar j}R_{i\bar qk\bar l}&=-T_{\bar q\bar j}^{\bar s}R_{i\bar sk\bar l}.
\end{align*}
\end{lma}

It can be checked easily that for $X,Y\in T^{1,0}M$, $R(X,\bar X, Y,\bar Y)$ is real-valued. We consider the following curvature condition.
\begin{defn}\label{BK-def}
We say that $(M,g)$ has Griffiths non-positive Chern curvature if there is non-positive function $\kappa(p)$ such that for any $p\in M$, $X,Y\in T^{1,0}_pM$,
$$R(X,\bar X,Y,\bar Y)\leq \kappa B(X,\bar X, Y,\bar Y) $$
where $B_{i\bar jk\bar l}=g_{i\bar j} g_{k\bar l}+g_{i\bar l}g_{k\bar j}$. We will denote it by $\mathrm{BK}\leq \kappa $ if $\kappa$ is a function.
\end{defn}

\begin{defn}
We say that $(M,g)$ has first Ricci curvature bounded above by a function $\kappa(x)$ if for any $p\in M$, $X\in T_p^{1,0}M$, 
$$Ric(X,\bar X)\leq \kappa(p) g(X,\bar X).$$
If $\kappa$ is non-positive and negative at some point $z\in M$, then we say that $g$ has quasi-negative first Ricci curvature.
\end{defn}
In this note, all the curvature tensor $\Rm$ will be referring to the curvature tensor with respect to Chern connection.

\section{evolution equation along the Hermitian curvature flow}
In this section, we will discuss a special type of Hermitian curvature flow introduced by \cite{StreetTian2011} with $Q\equiv 0$: 
\be\label{HRF}
\left\{
  \begin{array}{ll}
    \frac{\p}{\p t}g_\ijb=  &  -S_{i\bar j}; \\
    g(0)=  &g_0.
  \end{array}
\right.
\ee
Here $S_{i\bar j}=g^{k\bar l}R_{k\bar li\bar j}$ is the second Ricci curvature with respect to the Chern connection while the Chern-Ricci curvature (or first Ricci curvature) is defined by $R_{i\bar j}=g^{k\bar l}R_{i\bar jk\bar l}$. When the metric is K\"ahler, the first and second Ricci curvature coincides with the Riemannian Ricci curvature. However they are all different in general.

Now we derive some evolution equation for $R_{i\bar j k\bar l}$. This can be found in \cite[Section 6]{StreetTian2011}. We include the computation here for reader's convenience.
\begin{lma}Along the Hermitian curvature flow \eqref{HRF}, we have 
\begin{equation}\label{Rm-evo}
\begin{split}
\partial_t R_{i\bar j k\bar l}
&=\Delta R_{i\bar jk\bar l}+g^{r\bar s}\Big[T^p_{ri}\,\nabla_{\bar s} R_{p\bar jk\bar l}+T^{\bar q}_{\bar s \bar j}\,\nabla_r R_{i\bar qk\bar l}+T^p_{ri}T^{\bar q}_{\bar s\bar j}R_{p\bar q k\bar l}\\
&\quad  +R_{i\bar j r}\,^pR_{p\bar sk\bar l}+R_{r\bar j k}\,^p R_{i\bar s p\bar l} -R_{r\bar jp\bar l} R_{i\bar s k}\,^p\Big]\\
&\quad -\frac{1}{2}\left[S^p_i R_{p\bar jk\bar l}+S^p_k R_{i\bar jp\bar l}+S^{\bar q}_{\bar j}R_{i\bar qk\bar l}+S^{\bar q}_{\bar l}R_{i\bar j k\bar q} \right]
\end{split}
\end{equation}
where $\Delta=\Delta_{g(t)}$.
\end{lma}
\begin{proof}
\begin{equation}
\begin{split}\label{curv1}
\partial_t R_{i\bar j k}\,^l
&=-\partial_t \partial_{\bar j}\Gamma_{ik}^l\\
&=-\partial_{\bar j} \left(g^{l\bar q}\partial_i g_{k\bar q} \right)'\\
&=\partial_{\bar j} \left(g^{l\bar q} \partial_i S_{k\bar q}-S^{l\bar q} \Gamma_{ik}^p g_{p\bar q} \right)\\
&=\nabla_{\bar j}\nabla_i S_{k}^l\\
&=g^{r\bar s}\nabla_{\bar j}\nabla_i R_{r\bar s k}\,^l.
\end{split}
\end{equation}

By Lemma \ref{l-Chern-connection-1}, we have 
\begin{align*}
\nabla_i R_{r\bar s k}\,^l
&=\nabla_r R_{i\bar sk}\,^l+T_{ri}^p\, R_{p\bar s k}\,^l.
\end{align*}
And hence, 
\begin{equation}
\begin{split}\label{curv2}
\nabla_{\bar j} \nabla_i R_{r\bar sk}\,^l
&=\nabla_{\bar j}\nabla_r R_{i\bar sk}\,^l+\nabla_{\bar j}\left(T_{ri}^p\, R_{p\bar s k}\,^l\right)\\
&=\nabla_r \nabla_{\bar j} R_{i\bar sk}\,^l+R_{r\bar ji}\,^pR_{p\bar sk}\,^l+R_{r\bar jk}\,^p R_{i\bar s p }\,^l-R_{r\bar jp}\,^l R_{i\bar sk}\,^p\\
&\quad -R_{r\bar j}\,^{\bar q}\,_{\bar s}R_{i\bar q k}\,^l+(\partial_{\bar j}T^p_{ri}) R_{p\bar sk}\,^l+T^p_{ri}\, \nabla_{\bar j}R_{p\bar s k}\,^l\\
&=\nabla_r \nabla_{\bar s}R_{i\bar j k}\,^l+(\partial_r T_{\bar s\bar j}^{\bar q}) R_{i\bar qk}\,^l+T^{\bar q}_{\bar s \bar j}\nabla_r R_{i\bar qk}\,^l\\
&\quad +R_{r\bar ji}\,^pR_{p\bar sk}\,^l+R_{r\bar jk}\,^p R_{i\bar s p }\,^l-R_{r\bar jp}\,^l R_{i\bar sk}\,^p\\
&\quad -R_{r\bar j}\,^{\bar q}\,_{\bar s}R_{i\bar q k}\,^l+(\partial_{\bar j}T^p_{ri}) R_{p\bar sk}\,^l+T^p_{ri}\, \nabla_{\bar j}R_{p\bar s k}\,^l\\
&=\nabla_r \nabla_{\bar s}R_{i\bar j k}\,^l+T^p_{ri}\,\nabla_{\bar s} R_{p\bar jk}\,^l+T^{\bar q}_{\bar s \bar j}\,\nabla_r R_{i\bar qk}\,^l+T^p_{ri}T^{\bar q}_{\bar s\bar j}R_{p\bar q k}\,^l\\
&\quad  -R_{r\bar s}\,^{\bar q}\,_{\bar j}R_{i\bar qk}\,^l+R_{i\bar j r}\,^pR_{p\bar sk}\,^l+R_{r\bar j k}\,^p R_{i\bar s p}\,^l -R_{r\bar jp}\,^l R_{i\bar s k}\,^p.
\end{split}
\end{equation}

By  Lemma \ref{Ricci-id}, 
\begin{equation}
\begin{split}
\nabla_i \nabla_{\bar j} R_{r\bar sk}\,^l&=\nabla_{\bar j} \nabla_i R_{r\bar sk}\,^l-R_{i\bar j r}\,^p R_{p\bar sk}\,^l-R_{i\bar j k}\,^p R_{r\bar sp}\,^l\\
&\quad +R_{i\bar j}\,^{\bar q}\,_{\bar s}R_{r\bar qk}\,^l+R_{i\bar j p}\,^l R_{r\bar s k}\,^p.
\end{split}
\end{equation}

The result follows by combining this with \eqref{curv1}, \eqref{curv2} and  \eqref{HRF}.
\end{proof}

By tracing $k$ and $l$, we arrive at the evolution equation of the first Ricci curvature (or first Ricci curvature).  
\begin{lma}\label{Ricci-evo}Along the Hermitian curvature flow, we have the following evolution equation for the first Ricci curvature.
\begin{equation}
\begin{split}\partial_t R_{i\bar j}=&\Delta R_{i\bar j}+g^{r\bar s} (T^p_{ri} \nabla_{\bar s} R_{p\bar j}+T^{\bar q}_{\bar s\bar j}\nabla_r R_{i\bar q}+T^p_{ri}T^{\bar q}_{\bar s\bar j}R_{p\bar q})\\
&+R_{i\bar jk}\,^p R_{p}\,^k-\frac{1}{2}\left[S^p_i R_{p\bar j} +S^{\bar q}_{\bar j}R_{i\bar q} \right].
\end{split}
\end{equation}
\end{lma}
\begin{proof}
\begin{equation}
\begin{split}
\partial_t R_{i\bar j}&=\partial_t( g^{k\bar l}R_{i\bar jk\bar l})\\
&=S^{k\bar l}R_{i\bar jk\bar l}+g^{k\bar l}\partial_tR_{i\bar jk\bar l}\\
&=\Delta R_{i\bar j}+g^{r\bar s}\Big[T^p_{ri}\,\nabla_{\bar s} R_{p\bar j}+T^{\bar q}_{\bar s \bar j}\,\nabla_r R_{i\bar q}+T^p_{ri}T^{\bar q}_{\bar s\bar j}R_{p\bar q }\\
&\quad  +R_{i\bar j r}\,^pR_{p\bar s}+g^{k\bar l}R_{r\bar j k}\,^p R_{i\bar s p\bar l} -g^{k\bar l}R_{r\bar jp\bar l} R_{i\bar s k}\,^p\Big]\\
&\quad -\frac{1}{2}\left[S^p_i R_{p\bar j}+S^{\bar q}_{\bar j}R_{i\bar q} \right]\\
&=\Delta R_{i\bar j}+g^{r\bar s} (T^p_{ri} \nabla_{\bar s} R_{p\bar j}+T^{\bar q}_{\bar s\bar j}\nabla_r R_{i\bar q}+T^p_{ri}T^{\bar q}_{\bar s\bar j}R_{p\bar q})\\
&\quad+R_{i\bar jk}\,^p R_{p}\,^k-\frac{1}{2}\left[S^p_i R_{p\bar j} +S^{\bar q}_{\bar j}R_{i\bar q} \right].
\end{split}
\end{equation}
\end{proof}

\section{Preservation of curvature conditions}
In this section, we will adapt the idea by Liu in \cite{Liu2014} to the Hermitian curvature flow to show that the non-positivity of first Ricci curvature will be preserved for a short time. The main idea is to observe that the extra torsion term $T*\nabla \Rm$ and curvature coupled term $S*\Rm$ appeared in the evolution of $\Rm$ can be regarded as connection term which can be cancelled by appropriate choice of vector field extension when we apply maximum principle.

Before we give a proof on the preservation of curvature conditions, we first recall the doubling-time estimate for the Hermitian curvature flow in \cite{StreetTian2011}.

\begin{prop}\label{double-time}
Suppose $(M,g_0)$ is compact Hermitian manifold with 
$$\sup_{M }|\Rm_{g_0}|+|T_{g_0}|^2+|\nabla_{g_0} T_{g_0}|\leq K_0$$
for some $K_0>0$. Then there is $c_1(n)>0$ such that the Hermitian curvature flow has a short-time solution on $M\times [0,c_1K_0^{-1}]$ with initial metric $g_0$ and satisfies 
$$\sup_{M \times[0,c_1K_0^{-1}]}|\Rm_{g}|+|T_{}|^2+|\nabla_{} T_{}|\leq 2K_0$$
\end{prop}
\begin{proof}
By \cite[Corollary 7.4]{StreetTian2011}, we have the short-time existence to \eqref{HRF}. It remains to establish the doubling-time estimates. By \cite[Lemma 7.2, Lemma 6.1]{StreetTian2011}, the function $F=|\Rm|^2+|T|^4+|\nabla T|^2$ satisfies 
$$\heat F\leq 2c_0F^\frac{3}{2}$$
for some $c_0(n)>0$. Apply maximum principle to $F(x,t)$, we may conclude that for all $t\in [0,C_nK_0^{-1}]\cap [0,(c_0K)^{-1}]$, 
$$F(x,t)\leq \left(\frac{1}{K_0^{-1}-c_0t}\right)^2.$$
The assertion follows by choosing sufficiently small $c_1(n)>0$.
\end{proof}

\begin{prop}\label{preserve-Ric}Suppose $(M,g(t))$ is a solution to the Hermitian curvature flow \eqref{HRF} such that $g_0$ has Griffiths non-positive Chern curvature. There is $t_0>0,\;K>0$ such that for all $t\in [0,t_0]$, $\Rm(g(t))$ satisfies the following conditions.
\begin{enumerate}
 \item $Ric\leq 0$;
 \item $|R_{u\bar v x\bar x}|^2\leq (1+Kt)R_{u\bar u} R_{v\bar v}$ for all $x,u,v\in T^{1,0}M$, $|x|=1$;
\end{enumerate}
\end{prop}

To begin with, we consider 
$$R_{i\bar jk\bar l}^\epsilon=R_{i\bar jk\bar l}-\epsilon B_{i\bar jk\bar l}$$where $\epsilon$ is a small positive number and $B_{i\bar jk\bar l}=g_{i\bar j} g_{k\bar l}+g_{i\bar l}g_{k\bar j}$. Then Proposition \ref{preserve-Ric} is a direct consequence of the following Lemma by letting $\e\rightarrow 0$.

\begin{lma}\label{almost-preserve}
Under the assumption of Proposition \ref{preserve-Ric}, there is $\tau >0,K>0$ such that for any $\epsilon>0$, $t\in [0,\tau]$, the following holds.
\begin{enumerate}
\item[\bf(a)]$R_{i\bar j}^\e< -\e e^{-Kt}g_{i\bar j}$ where $R_{i\bar j}^\e=g^{k\bar l}R^\e_{i\bar jk\bar l}$;
\item[\bf (b)] $|R^\e_{u\bar v x\bar x }|^2< (1+Kt)R^\e_{u\bar u}R^\e_{v\bar v}$ for all $x,u,v\in T^{1,0}M$, $|x|=1$.
\end{enumerate}
\end{lma}

Before we give a proof of Lemma \ref{almost-preserve}, we first show that $g(0)$ satisfies the assumptions in the Lemma. 
\begin{lma}\label{curv-zero}
Under the assumption of Proposition \ref{preserve-Ric}, $\Rm^\epsilon(0)$ satisfies 
\begin{enumerate}
\item$R_{i\bar j}^\e(g_0)< -\e (g_0)_{i\bar j}$ where $R_{i\bar j}^\e=g^{k\bar l}R^\e_{i\bar jk\bar l}$;
\item$|R(g_0)^\epsilon_{u\bar vx\bar x}|^2< R(g_0)^\epsilon_{u\bar u} R(g_0)^\epsilon_{v\bar v}$ for all $x,u,v\in T^{1,0}M$, $|x|=1$.
\end{enumerate}
\end{lma}
\begin{proof}
The conclusion on $Ric^\epsilon$ follows immediately from $Ric(g_0)\leq 0$ and $g^{k\bar l}B_{i\bar jk\bar l}=n+1$.

To check $(2)$, for a fixed $x\in T^{1,0}M$, $R^\epsilon_{ p\bar qx\bar x}$ is Hermitian form and hence we may choose eigenvectors $\{e_i\}_{i=1}^n$ such that 
$$R^\epsilon _{i\bar jx\bar x }=\lambda_i \delta_{ij}$$
where $\lambda_i<0$ as $\text{BK}(g_0)\leq 0$. Therefore, for $u=\sum_{i=1}^n u^ie_i$ and  $v=\sum_{i=1}^n v^i e_i$,
\begin{equation}
\begin{split}
|R^\epsilon_{u\bar vx\bar x }|^2
&=|\sum_{i=1}^n\lambda_iu^i v^{\bar i}|^2\\
&\leq \left[ \sum_{i=1}^n \lambda_i|u^i|^2\right]\left[ \sum_{i=1}^n \lambda_i|v^i|^2\right]\\
&=R^\epsilon_{ u\bar ux\bar x}\, R^\epsilon_{v\bar vx\bar x}\\
&< R_{u\bar u}^\epsilon \,R^\epsilon_{v\bar v}
\end{split}
\end{equation}
since $\mathrm{BK}(g_0)\leq 0$.
\end{proof}

\begin{proof}[Proof of Lemma \ref{almost-preserve}]
Denote $K_0=\sup_{M\times\{0\}}|\Rm|+|T|^2+|\nabla T|$ and choose $K>> K_0$ and $\tau =K^{-1}$. By proposition \ref{double-time}, the Hermitian curvature flow $g(t)$ exists on $M\times [0,\tau]$ with 
\begin{align}\label{doubletime}
\sup_{M \times [0,\tau]}|\Rm|+|T|^2+|\nabla T|\leq 2K_0.
\end{align}

Suppose condition {\bf (a)} and {\bf (b)} are true on $[0,\tau]$, then we are done. Let $t_0\in (0,\tau]$ be the first time such that one of them fails. \noindent\\

{\bf Case 1:} Condition {\bf (a)} is  true on $[0,t_0)$ and fails at $t=t_0$. Then there is $p\in M$, $X_0\in T^{1,0}_p M$ with $|X_0|=1$ such that 
$$Ric^\e(X_0,\bar X_0)=-\e e^{-Kt}.$$
Moreover, for all $z\in M$, $t\in [0,t_0]$, $y,u,v\in T_z^{1,0}M$ with $|Y|=1$,
\begin{equation}\label{a1}
\begin{split}
Ric^\e(Y,\bar Y)&\leq -\e  e^{-Kt};\\
|R^\e_{u\bar v y\bar y }|^2&< (1+Kt)R^\e_{u\bar u}R^\e_{v\bar v}.
\end{split}
\end{equation}

As in \cite[Page 1599]{Liu2014}, we may use polarization and \eqref{doubletime} to infer that for sufficiently small $\e>0$, any $e_k,e_l\in T^{1,0}$ with unit $1$ and $e_i,e_j\in T^{1,0}$,
\begin{align}\label{a2}
|R^\e_{i\bar jk\bar l}|^2 \leq C_n R^\e_{i\bar i}R^\e_{j\bar j}, \;\; |R^\e_{i\bar jk\bar l}|^2 \leq C_nK_0 |R^\e_{i\bar i}|.
\end{align}
\noindent

Consider the following tensor $$A_{i\bar j}=R_{i\bar j}+\e\left[e^{-Kt}-(n+1) \right] g_{i\bar j}=Ric^\e_{i\bar j}+\e e^{-Kt}g_{i\bar j}$$
which satisfies $A(X_0,\bar X_0)=0$ and $A(Y,\bar Y)\leq 0$ for all $Y\in T^{1,0}_xM$, $x\in M$. We may assume $|X_0|_{g(t_0)}=1$ by rescaling.

Extend $X_0$ locally to a vector field around $(p,t_0)$ { such that at $(p,t_0)$,
\begin{equation}
\label{ext-1}
\begin{split}
\nabla_{\bar q} X^p&=0;\\
\nabla_{p}X^q &=T^q_{pl}\;X^l.
\end{split}
\end{equation}
}

Then $A(X,\bar X)$ locally defined a function and satisfies 
\begin{align}\label{max-1}
\Box A(X,\bar X) \geq 0.
\end{align}
where we denote $\heat$ by $\Box$ for notational convenience.

Now we compute the evolution equation for $A(X,\bar X)$. At $(p,t_0)$, 
\begin{equation}\label{equ-A-1}
\begin{split}
\frac{\partial}{\partial t} A(X,\bar X)
&=\left(\partial_t A_{i\bar j}\right) X^iX^{\bar j} +A_{i\bar j} \left( \partial_t X^i X^{\bar j} +X^i \partial_t X^{\bar j}\right)\\
&=\left(\partial_t R_{i\bar j}-\e[e^{-Kt}-(n+1)]S_{i\bar j}-\e Ke^{-Kt}g_{i\bar j}\right)X^i X^{\bar j}\\
&\quad +A_{i\bar j} \left( \partial_t X^i \cdot X^{\bar j} +X^i\cdot \partial_t X^{\bar j}\right)\\
&\leq (\partial_t R_{i\bar j}) X^i X^{\bar j} -\frac{1}{2} \e K e^{-Kt}.
\end{split}
\end{equation}
Here we have used  \eqref{doubletime} and  the fact that for any $Y\in T^{1,0}_pM$,
\begin{align}\label{firstorder}
A_{X_0\bar Y}=0
\end{align}

Now we compute the $\Delta A(X,\bar X)$. We may in addition assume that at $(p,t_0)$, $g_{i\bar j}=\delta_{i\bar j}$. For notational convenience for a tensor $\a$, we will use $\a_{;X}$ to denote covariant derivative $\nabla_X \a$. Then 
\begin{equation}\label{equ-A-2}
\begin{split}
\Delta A(X,\bar X)
&=\frac{1}{2} g^{r\bar s} (\nabla_r \nabla_{\bar s}+\nabla_{\bar s}\nabla_r) \Big( A_{i\bar j }X^i X^{\bar j}\Big)\\
&=\Delta A_{i\bar j}\cdot X^i X^{\bar j}+A_{i\bar j} X^i \Delta X^{\bar j}+A_{i\bar j} X^{\bar j} \Delta X^{ i}\\
&\quad +A_{i\bar j;r} X^i_{;\bar r}X^{\bar j} +A_{i\bar j;\bar r} X^i_{; r}X^{\bar j} \\
&\quad +A_{i\bar j;r} X^iX^{\bar j}_{;\bar r} +A_{i\bar j;\bar r} X^iX^{\bar j}_{; r} \\
&\quad +A_{i\bar j} X^i_{;r} X^{\bar j}_{;\bar r}+A_{i\bar j} X^i_{;\bar r} X^{\bar j}_{; r}\\
&=\Delta R_{i\bar j}\cdot X^i X^{\bar j}+R_{i\bar j;r} X^i_{;\bar r}X^{\bar j} +R_{i\bar j;\bar r} X^i_{; r}X^{\bar j} \\
&\quad +R_{i\bar j;r} X^iX^{\bar j}_{;\bar r} +R_{i\bar j;\bar r} X^iX^{\bar j}_{; r} \\
&\quad +A_{i\bar j} X^i_{;r} X^{\bar j}_{;\bar r}+A_{i\bar j} X^i_{;\bar r} X^{\bar j}_{; r}\\
\end{split}
\end{equation}
Here we have also used $\nabla g=0$ and \eqref{firstorder} on the terms involving $A(\Delta X,\bar X)$ or its conjugate.

By combining \eqref{equ-A-1} and \eqref{equ-A-2}, we have
\begin{equation}
\begin{split}
\Box A(X,\bar X)
&\leq X^iX^{\bar j}\Box R_{i\bar j} -\frac{1}{2}\e K e^{-Kt}\\
&\quad -g^{r\bar s}\Big[R_{i\bar j;r} X^i_{;\bar s}X^{\bar j} +R_{i\bar j;\bar s} X^i_{; r}X^{\bar j} \\
&\quad +R_{i\bar j;r} X^iX^{\bar j}_{;\bar s} +R_{i\bar j;\bar s} X^iX^{\bar j}_{; r} \Big]\\
&\quad -g^{r\bar s}\left(R_{i\bar j} X^i_{;r} X^{\bar j}_{;\bar s}+R_{i\bar j} X^i_{;\bar s} X^{\bar j}_{; r}\right)\\
&\quad -\e\Big[ e^{-Kt}-(n+1)\Big] g^{r\bar s} \left( X^i_{;r}X^{\bar j}_{;\bar s}+X^i_{;\bar s}X^{\bar j}_{;r}\right).
\end{split}
\end{equation}

By \eqref{ext-1} and \eqref{doubletime}, then it reduces to 
\begin{equation}
\begin{split}
\Box A(X,\bar X)&\leq R_{X\bar Xk}\;^p R_p\;^k-\frac{1}{2}(S^p_X R_{p\bar X} +S^{\bar q}_{\bar X}R_{X\bar q})
-\frac{1}{4}\e K.
\end{split}
\end{equation}

As $K>>K_0$, by \eqref{a1},\eqref{firstorder},\eqref{a2} and  \eqref{doubletime}, we have 
\begin{equation}
\begin{split}
\Box A(X,\bar X)&\leq R_{X\bar Xk}\;^p R_p\;^k
-\frac{1}{8}\e K\\
&\leq 
-\frac{1}{10}\e K
\end{split}
\end{equation}
which contradicts with \eqref{max-1}. 

\noindent\\

{\bf Case 2:} Suppose condition {\bf (b)} is not true at $t=t_0$. Then there is $p\in M$, $x_0,u_0,v_0\in T^{1,0}_pM$ with $|x_0|_{t_0}=1$ such that 
$$|R^\e_{u_0\bar v_0 x_0\bar x_0}|^2=(1+Kt) R^\e_{u_0 \bar u_0}R^\e_{v_0\bar v_0}.$$
By rescaling, we may assume $|u_0|_{t_0}=|v_0|_{t_0}=1$. Moreover for all $(z,t)\in M\times [0,t_0]$, $x,u,v\in T_z^{1,0}M$ with $|x|_t=1$,
\begin{equation}\label{b1}
\begin{split}
R^\e_{x\bar x}&\leq -\e e^{-Kt};\\
|R^\e_{u\bar v x\bar x}|^2&\leq (1+Kt) R^\e_{u \bar u}R^\e_{v\bar v}.
\end{split}
\end{equation}

Argue as in \eqref{a2}, we know that for any $e_k,e_l\in T^{1,0}$ with unit $1$, 
\begin{align}\label{b2}
|R^\e_{i\bar jk\bar l}|^2 \leq C_n R^\e_{i\bar i}R^\e_{j\bar j}, \;\; |R^\e_{i\bar jk\bar l}|^2 \leq C_nK_0 |R^\e_{i\bar i}|.
\end{align}

As in case 1, we will extend $x_0$, $u_0$ and $v_0$ to local vector fields so that we may apply maximum principle. {We extend $x_0,u_0,v_0$ to $X$, $U$ and $V$ such that at $(p,t_0)$,
\begin{equation}\label{ext-2}
\begin{split}
\nabla_{\bar s} U^r=0 &, \quad \nabla_{p} U^r=T^r_{p q}U^q,\quad \Box U^r=\frac{1}{2}S^r_p U^p;\\
 \nabla_{\bar s} V^r=0& ,\quad \nabla_{p} V^r=T^r_{p q}V^q;\quad  \Box V^r=\frac{1}{2}S^r_p V^p;\\
 \nabla_{\bar s} X^r=0& ,\quad \nabla_{p} X^r=0;\quad\quad\quad   \Box X^r=0.\\
\end{split}
\end{equation}
}
Hence the function $$F(x,t)=g_{X\bar X}^{-2}|R^\e_{U\bar V X\bar X}|^2-(1+Kt)R^\e_{U\bar U}R^\e_{V\bar V}$$ is defined locally around $(p,t_0)$ and attains its local maximum at $(p,t_0)$ and therefore satisfies
\begin{align}\label{max-2}
\Box F\Big|_{(p,t_0)}\geq 0.
\end{align}

We now differentiate each of them carefully. Using \eqref{ext-2} and Lemma \ref{Ricci-evo}, a similar calculation as in {\bf case 1} yields 
\begin{equation}\label{evo-Ric-1}
\begin{split}
\Box R^\e_{U\bar U}
&=\Box R^\e_{i\bar j}\cdot U^i U^{\bar j}+R^\e_{i\bar j} U^i \Box U^{\bar j} +R^\e_{i\bar j}\Box U^i\cdot  U^{\bar j}\\
&\quad -g^{r\bar s}R^\e_{i\bar j} U^i_{;r} U^{\bar j}_{;\bar s}
-g^{r\bar s}R_{i\bar j;r}  U^iU^{\bar j}_{;\bar s}
-g^{r\bar s}R_{i\bar j;\bar s} U^i_{;r} U^{\bar j}\\
&=\Box R_{i\bar j}\cdot U^i U^{\bar j} +\e(n+1)S_{U\bar U}\\
&\quad +R^\e_{U\bar j} \Box U^{\bar j} +R^\e_{i\bar U}\Box U^i+\e(n+1)g^{r\bar s}g_{i\bar j} T^i_{rp} T^{\bar j}_{\bar s\bar q}U^p U^{\bar q}\\
&\quad -g^{r\bar s}R_{i\bar j} T^i_{rp} T^{\bar j}_{\bar s\bar q}U^pU^{\bar q}
-g^{r\bar s}R_{i\bar j;r}  T^{\bar j}_{\bar s\bar q}U^iU^{\bar q}
-g^{r\bar s}R_{i\bar j;\bar s} T^i_{rp}U^p U^{\bar j}\\
&= R_{U\bar U k}\;^p R^k_p-\e(n+1)S_{U\bar U} +\e(n+1)g^{r\bar s}g_{i\bar j} T^i_{rU} T^{\bar j}_{\bar s\bar U}.
\end{split}
\end{equation}
Similarly, 
\begin{equation}\label{evo-Ric-2}
\begin{split}
\Box R^\e_{V\bar V}
&= R_{V\bar V k}\;^p R^k_p-\e(n+1)S_{V\bar V}+\e(n+1)g^{r\bar s}g_{i\bar j} T^i_{rV} T^{\bar j}_{\bar s\bar V}.
\end{split}
\end{equation}

By combining \eqref{evo-Ric-1}, \eqref{evo-Ric-2} with \eqref{b1}, \eqref{b2}, \eqref{doubletime} and using the fact that $K>>K_0$ and $\tau=K^{-1}$, we arrive at the following inequality.
\begin{equation}
\begin{split}
&\quad \Box \left[(1+Kt)R^\e_{U\bar U}R^\e_{V\bar V}\right]\\
&\geq (1+Kt)R^\e_{U\bar U} \left( R^\e_{V\bar V k}\;^p R^k_p\right) +(1+Kt)R^\e_{V\bar V}\left( R^\e_{U\bar U k}\;^p R^k_p\right)\\
&\quad -2(1+Kt){\bf Re}\left(g^{r\bar s}\nabla_r R^\e_{U\bar U} \cdot \nabla_{\bar s}R^\e_{V\bar V} \right)+\frac{1}{2}KR^\e_{U\bar U}R^\e_{V\bar V}\\
&\geq -2(1+Kt){\bf Re}\left(g^{r\bar s}\nabla_r R^\e_{U\bar U} \cdot \nabla_{\bar s}R^\e_{V\bar V} \right)+\frac{1}{4} KR^\e_{U\bar U }R^\e_{V\bar V}
\end{split}
\end{equation}

\noindent\\
Now we derive the evolution equation of $|R^\e_{U\bar VX\bar X}|^2$. Similar to the computation of $\Box R^\e_{U\bar U}$, using \eqref{ext-2} and Lemma \ref{Rm-evo}, we have 
\begin{equation}\label{evo-Rm-1}
\begin{split}
\Box R^\e_{U\bar VX\bar X}
&=\Box R_{i\bar jk\bar l}\cdot U^iV^{\bar j}X^kX^{\bar l}-\e \Box B_{i\bar jk\bar l}\cdot U^iV^{\bar j}X^kX^{\bar l}\\
&\quad + R^\e_{i\bar jk\bar l}(\Box U^i) V^{\bar j}X^kX^{\bar l}+ R^\e_{i\bar jk\bar l} U^i(\Box V^{\bar j})X^kX^{\bar l}\\
&\quad +R^\e_{i\bar jk\bar l} U^iV^{\bar j}\left(\Box X^k \;X^{\bar l}+X^k \Box X^{\bar l}\right)\\
&\quad -g^{r\bar s} R_{i\bar jk\bar l;r}\left(U^iV^{\bar j}_{;\bar s}X^kX^{\bar l}+U^i V^{\bar j}X^k X^{\bar l}_{;\bar s} \right)\\
&\quad -g^{r\bar s} R_{i\bar jk\bar l;\bar s}\left(U^i_{;r}V^{\bar j}X^kX^{\bar l}+U^i V^{\bar j}X_{;r}^k X^{\bar l} \right)\\
&\quad  -R^\e_{i\bar jk\bar l} \Big(U^i_{;r}V^{\bar j}_{;\bar s}X^kX^{\bar l} +U^iV^{\bar j}_{;\bar s}X^k_{;r}X^{\bar l} \\
&\quad +U^i_{;r}V^{\bar j}X^kX^{\bar l}_{;\bar s}+U^iV^{\bar j}X^k_{;r}X^{\bar l}_{;\bar s}  \Big)\\
&= g^{r\bar s}\Big[R_{U\bar V r}\,^pR_{p\bar sX\bar X}+R_{r\bar V X}\,^p R_{U\bar s p\bar X} -R_{r\bar Vp\bar X} R_{U\bar s X}\,^p\Big] \\
&\quad -\frac{1}{2}\left[S^p_X R_{U\bar Vp\bar X}+S^{\bar q}_{\bar X}R_{U\bar V X\bar q} \right]\\
&\quad +\e \left(S_{U\bar V}g_{X\bar X}+g_{U\bar V}S_{X\bar X}+S_{U\bar V}g_{X\bar X}+g_{U\bar V}S_{X\bar X}\right)\\
&\quad +\e B_{i\bar jX\bar X} T^i_{rU}T^{\bar j}_{\bar s\bar V}-\frac{1}{2} \e\left(B_{i\bar VX\bar X}S^i_U +B_{U\bar j X\bar X}S^{\bar j}_{\bar V} \right).
\end{split}
\end{equation}

Similarly, 
\begin{equation}\label{evo-Rm-2}
\begin{split}
\Box R^\e_{V\bar UX\bar X}
&=g^{r\bar s}\Big[R_{V\bar U r}\,^pR_{p\bar sX\bar X}+R_{r\bar U X}\,^p R_{V\bar s p\bar X} -R_{r\bar Up\bar X} R_{V\bar s X}\,^p\Big] \\
&\quad -\frac{1}{2}\left[S^p_X R_{V\bar Up\bar X}+S^{\bar q}_{\bar X}R_{V\bar U X\bar q} \right]\\
&\quad +\e \left(S_{V\bar U}g_{X\bar X}+g_{V\bar U}S_{X\bar X}+S_{V\bar U}g_{X\bar X}+g_{V\bar U}S_{X\bar X}\right)\\
&\quad +\e B_{i\bar jX\bar X} T^i_{rV}T^{\bar j}_{\bar s\bar U}-\frac{1}{2} \e\left(B_{i\bar UX\bar X}S^i_V +B_{V\bar j X\bar X}S^{\bar j}_{\bar U} \right).
\end{split}\end{equation}

We emphasis that $U$ and $V$ only appear in the first two entry of the Chern-curvature tensor. Therefore by combining \eqref{evo-Rm-1}, \eqref{evo-Rm-2} \eqref{doubletime}, \eqref{b2} and \eqref{b1}, we can show that
\begin{equation}
\begin{split}
\Box (R^\e_{U\bar VX\bar X}R^\e_{V\bar UX\bar X})
&\leq -|\nabla R^\e_{U\bar VX\bar X}|^2-|\bar \nabla R^\e_{U\bar VX\bar X}|^2+C_nK_0 R^\e_{U\bar U}R^\e_{V\bar V}.
\end{split}
\end{equation}
And hence at $(p,t_0)$, 
\begin{equation}
\begin{split}
\heat F&\leq 2(1+Kt){\bf Re}\left(g^{r\bar s}\nabla_r R^\e_{U\bar U} \cdot \nabla_{\bar s}R^\e_{V\bar V} \right)\\
&\quad  -|\nabla R^\e_{U\bar VX\bar X}|^2-|\bar \nabla R^\e_{U\bar VX\bar X}|^2\\
&\quad -\frac{1}{8}K R^\e_{U\bar U}R^\e_{V\bar V}+2S_{X\bar X}|R^\e_{U\bar VX\bar X}|^2
\end{split}
\end{equation}

By using the fact that $\nabla F=0$ and $F=0$ at $(p,t_0)$, one can conclude that
$$2(1+Kt){\bf Re}\left(g^{r\bar s}\nabla_r R^\e_{U\bar U} \cdot \nabla_{\bar s}R^\e_{V\bar V} \right) \leq |\nabla R^\e_{U\bar VX\bar X}|^2+|\bar \nabla R^\e_{U\bar VX\bar X}|^2.$$

Using \eqref{doubletime} and $F(p,t_0)=0$ again, we deduce that $$2S_{X\bar X}|R^\e_{U\bar VX\bar X}|^2\leq C_nK_0R^\e_{U\bar U}R^\e_{V\bar V}$$ and hence at $(p,t_0)$, 
\begin{align}
\heat F<-\frac{1}{16}K R^\e_{U\bar U} R^\e_{V\bar V}
\end{align}
which contradicts with \eqref{max-2} provided that $K>\tilde C_nK_0$ for some $\tilde C_n>>1$.
\end{proof}

\begin{rem}
From the proof, the Griffiths non-positivity of Chern curvature is only used to ensure that $\Rm$ satisfies the curvature condition at $t=0$. It is clear that the proof is still true if we replace $\mathrm{BK}\leq 0$ by $R_{u\bar u}g_{v\bar v}\leq \lambda R_{u\bar u v\bar v}$ for some $\lambda<n^{-1}$ and modify $(1+Kt)$ to $(a+Kt)$ for some $a(n,\lambda)$.
\end{rem}

\section{Strong Maximum Principle}

In this section, we will show that the first Ricci curvature will become strictly negative shortly after the Hermitian curvature flow evolves. We will adapt the strong maximum principle in \cite{Cao1997} to the Hermitian curvature flow setting. One may also consider the Kernel of the first Ricci curvature in the Hermitian setting, see \cite[Theorem 5.2]{Yury2016}.

\begin{thm}\label{SMP-Ric}
Suppose $(M,g(t))$ is a solution to the Hermitian curvature flow on $M\times [0,T]$ with initial metric $g_0$. If $g_0$ has Griffiths non-positive Chern curvature and its first Ricci curvature is negative at some $p\in M$, then there is $\tau>0$ such that $\Ric(g(t))<0$ on $(0,\tau]$.
\end{thm}

\begin{proof}
Let $\tau$ be the constant obtained from Proposition \ref{preserve-Ric}. We adopt the argument in \cite{Cao1997} to Hermitian curvature flow setting. Let $y\in M$ be a point at which the first Ricci curvature is negative. Let $\phi_0$ be a smooth non-negative function such that $\phi_0(y)>0$, $\phi_0=0$ outside a neighbour of $y$ and $$\Ric(g_0)+\phi_0g_0\leq 0$$ on $M$. Let $\phi(z,t)$ be the solution to the heat equation
\begin{equation}
\label{heat}
\begin{split}
\left( \frac{\partial}{\partial t}-\Delta_{g(t)}\right) \phi(x,t)&=0,\quad\text{on}\;\;M\times [0,\tau];\\
\phi(x,0)&=\phi_0.
\end{split}
\end{equation}
It then follows by strong maximum principle that $\phi(x,t)>0$ on $M\times(0,\tau]$. We may assume that $\phi(x,t)\leq 1$ by rescaling.

Let $k=c_nK_0$ with $c_n>>1$. For any $\e>0$, consider the tensor 
$$A^\e=\Ric_{g(t)}+e^{-kt}\phi^2 g(t)-\e e^{Bt}g(t)$$
where $B$ is some large constant to be specified later. We  claim that $A^\e\leq 0$ on $M\times [0,\tau]$. Then the result follows by letting $\e\rightarrow 0$. We omit the index $\e$ for notational convenience. Noted that $A(0)<0$ on $M$. Suppose not, there is $t_0\in (0,\tau]$ such that for all $(z,t)\in M\times [0,t_0]$, $u\in T^{1,0}_zM$, 
\begin{equation}
\begin{split}
A_{u\bar u}(z,t)&\leq 0.
\end{split}
\end{equation}
And there is $x_0\in M$, $v\in T^{1,0}_xM$ so that $A_{v\bar v}(x_0,t_0)=0$. We may further assume that $|v|_{t_0}=1$ by rescaling. As in the proof of Lemma \ref{almost-preserve}, we extend $v$ around $(x_0,t_0)$ locally such that 
\begin{equation}\label{ext-3}
\begin{split}
\nabla_{\bar q}v^p=0,\;\;\;\nabla_{ q}v^p=T^p_{ql}v^l.
\end{split}
\end{equation}

Therefore, the function $A_{v\bar v}$ attains local maximum at $(x_0,t_0)$ and therefore obeys 
\begin{align}\label{maxpoint}
\heat\Big|_{(x_0,t_0)} A_{v\bar v}\geq 0.
\end{align}

On the other hand, if we choose coordinate at $x_0$ such that $g_{i\bar j}=\delta_{i\bar j}$, then 
\begin{equation}\label{evo-A-1}
\begin{split}
\heat A_{v\bar v}&=\Box A_{i\bar j}\cdot v^iv^{\bar j}+A_{i\bar j}\left(v^i \Box v^{\bar j}+ v^{\bar j}\Box v^i \right)\\
&\quad -g^{r\bar s}\left(A_{i\bar j}v^i_{;r} v^{\bar j}_{;\bar s}+A_{i\bar j;r} v^i v^{\bar j}_{\bar s}+A_{i\bar j;\bar s}v^i_{;r}v^{\bar j}\right)\\
&=\Box A_{i\bar j} v^i v^{\bar j} -A_{i\bar j} T^i_{rv} T^{\bar j}_{\bar r\bar v}-A_{i\bar j;\bar r}T^i_{rv}v^{\bar j}-A_{i\bar j; r}T^{\bar j}_{\bar r\bar v}v^{i}.
\end{split}
\end{equation}
where we have used \eqref{ext-3}. Moreover, the first bracket vanishes due to the fact that 
\begin{align}\label{first-order-2}
A_{v\bar u}=0
\end{align}
 for all $u\in T^{1,0}_{x_0}$. This fact can be seen by considering the first variation of functions, $A(v+tu,\bar v+t\bar u)$ and $A(v+t\sqrt{-1}u,\bar v-t\sqrt{-1}\bar u)$ at $t=0$.

Moreover, we use Lemma \ref{Ricci-evo}, \eqref{first-order-2} and \eqref{heat} to deduce that 
\begin{equation}
\begin{split}
\Box A_{i\bar j}\cdot v^iv^{\bar j}&=R_{i\bar j} T^i_{rv} T^{\bar j}_{\bar r\bar v}+R_{i\bar j;\bar r}T^i_{rv}v^{\bar j}+R_{i\bar j; r}T^{\bar j}_{\bar r\bar v}v^{i}\\
&\quad +R_{v\bar vk}\,^p R_{p}^k -\frac{1}{2} \left(S_v^pR_{p\bar v} +S^{\bar q}_{\bar v} R_{v\bar q} \right)\\
&\quad -ke^{-kt} \phi^2 -2|\nabla\phi|^2 e^{-kt} -\phi^2 e^{-kt} S_{v\bar v}\\
&\quad -\e Be^{Bt} +\e e^{Bt}S_{v\bar v}
\end{split}
\end{equation}

Combines with \eqref{evo-A-1} together with \eqref{doubletime}, \eqref{first-order-2}, Theorem \ref{preserve-Ric} and the fact that $A_{v\bar v}(x_0,t_0)=0$, it gives
\begin{equation}
\begin{split}
\heat A_{v\bar v}&=\left(-\phi^2e^{-kt}+\e e^{Bt} \right)g_{i\bar j}T^i_{rv}T^{\bar j}_{\bar r\bar v}-4\phi e^{-kt} {\bf Re}\left( \phi_{\bar r} T_{rv\bar v}\right)\\
&\quad +R_{v\bar vp\bar q}R^{\bar q p} -ke^{-kt} \phi^2 -2e^{-kt} |\nabla\phi|^2-\e Be^{Bt}\\
&\leq (-k+C_nK_0)\phi^2 e^{-kt}+ \e e^{Bt}( -B+C_nK_0).
\end{split}
\end{equation}
Hence, if we choose $k$ and $B$ sufficiently large, then it contradicts with \eqref{maxpoint}. In conclusion, we have shown that there is $k,B>0$ such that for all $\e>0$, $(x,t)\in M\times [0,\tau]$, 
$$\Ric(g(t))\leq \left(-e^{-kt}\phi(x,t)+\e e^{Bt}\right) g.$$
In particular, by letting $\e\rightarrow 0$, we have $\Ric(g(t))<0$ when $t\in (0,\tau]$.
\end{proof}

By using Theorem \ref{SMP-Ric}, the main theorem is immediate.
\begin{proof}[Proof of Corollary \ref{main-2}]
By the short time existence result in \cite{StreetTian2011}, there is a short time solution to the Hermitian curvature flow \eqref{HRF} starting from $g_0$. By Theorem \ref{preserve-Ric} and Theorem \ref{SMP-Ric}, there is $\tau>0$ such that $\Ric(g(\tau))<0$ on $M$. Therefore, $c_1(K_M)>0$ and hence K\"ahler. The existence of K\"ahler-Einstein follows by a fundamental result of Aubin \cite{Aubin1976} and Yau \cite{Yau1978}, see also \cite{TosattiWeinkove2015} for a parabolic proof using the Chern-Ricci flow starting from any smooth Hermitian metric.
\end{proof}

\end{document}